\documentclass{article}
\usepackage{amsfonts}
\usepackage{amsmath}
\usepackage{theorem}
\usepackage{array}
\usepackage{graphics}
\usepackage[usenames]{color}
\usepackage[a4paper]{geometry}
\usepackage{amssymb}
\usepackage[all]{xypic}
\usepackage{graphicx}
\newtheorem{theorem}{Theorem}

\newtheorem{corollary}[theorem]{Corollary}

\newtheorem{definition}[theorem]{Definition}

{ \theorembodyfont{\rmfamily}
\newtheorem{example}[theorem]{Example}

\newtheorem{remark}[theorem]{Remark}
}

\newtheorem{lemma}[theorem]{Lemma}

\newtheorem{proposition}[theorem]{Proposition}

\newenvironment{proof}[1][Proof]{\textbf{#1.} }{\ \rule{0.5em}{0.5em}}

\newcommand{\F}{\mathbb{F}_q}
\newcommand{\FF}{\mathbb{F}}

\newcommand{\Q}{\mathbb{Q}}

\newcommand{\PGL}{{\mathrm{PGL}}}

\begin{document}

\title{Explicit equations for Drinfeld modular towers}
\author{Alp Bassa and Peter Beelen}
\date{}
\maketitle
\abstract{Elaborating on ideas of Elkies, we show how recursive equations for towers of Drinfeld modular curves $(X_0(P^n))_{n\ge 0}$ for $P\in \mathbb F_q[T]$ can be read of directly from the modular polynomial $\Phi_P(X,Y)$ and how this naturally leads to recursions of depth two. Although the modular polynomial $\Phi_T(X,Y)$ is not known in general, using generators and relations given by Schweizer, we find unreduced recursive equations over $\mathbb F_q(T)$ for the tower $(X_0(T^n))_{n\ge 2}$ and of a small variation of it (its partial Galois closure). Reducing at various primes, one obtains towers over finite fields, which are optimal, i.e., reach the Drinfeld--Vladut bound, over a quadratic extension of the finite field. We give a proof of the optimality of these towers, which is elementary and does not rely on their modular interpretation except at one point. We employ the modular interpretation to determine the splitting field of certain polynomials, which are analogues of the Deuring polynomial. For these towers, the particular case of reduction at the prime $T-1$ corresponds to towers introduced by Elkies and Garcia--Stichtenoth.}

\section{Introduction}

The question of how many rational points a curve of genus $g$ defined
over a finite field $\mathbb F_q$ can have, has been a central and
important one in number theory. One of the landmark results in the
theory of curves defined over finite fields was the theorem of Hasse
and Weil, which is the congruence function field analogue of the
Riemann hypothesis. As an immediate consequence of this theorem one
obtains an upper bound for the number of rational points on such a
curve in terms of its genus and the cardinality of the finite field.
It was noticed however by Ihara \cite{ihara} and Manin \cite{manin1}
that this bound can be improved for large genus and the asymptotic
study over a fixed finite field was then initiated by Ihara. An
asymptotic upper bound on the number of rational points was given by
Drinfeld and Vladut \cite{dv}.

Finding curves of large genera with many points is a difficult task
and there have basically been three approaches: class field theory
(see among others \cite{xingnied,serre}), explicit constructions (see
among others \cite{elkies,GSINV,GSJNT,tame}) and reductions of modular
curves of various types (see among others
\cite{ihara,ge2,TVZ,manin}). With these techniques it is possible
to construct sequences of curves having many points compared to their
genera asymptotically and in some cases even attaining the Drinfeld--Vladut bound, in which case the sequence of curves is called optimal.

In \cite{GSJNT}, Garcia and Stichtenoth introduced the following
optimal sequence of function fields $(F_n)_{n \ge 0}$ over $\mathbb F_
\ell$, where $\ell=q^2$ : Let $F_0=\mathbb F_\ell(x_0)$  and define
$F_{n+1}=F_n(x_{n+1})$ where
$$x_{n+1}^q+x_{n+1}=\frac{x_n^q}{x_n^{q-1}+1},$$
for $n\geq 0$.
Because of the way the tower is defined, it is said to be recursively
given by the equation
\begin{equation}\label{eq:jnteq}
y^q+y=\frac{x^q}{x^{q-1}+1}.
\end{equation}
In \cite{elkies,elkiesd}, Elkies gave a modular interpretation for
this and for all other known optimal recursive towers. More precisely
he showed that all known examples of tame, (respectively wild) optimal
recursive towers correspond to reductions of classical (respectively
Drinfeld) modular curves. Moreover, he found several other equations
for such towers, by studying  reductions of Drinfeld-, elliptic- an
Shimura-modular curves very explicitly and gave an explanation for the
recursive nature of these towers. Among other things, he showed that the
reduction of the tower of Drinfeld modular curves $(X_0(T^n))_{n\geq
2}$ at the prime $T-1$ is given recursively by the equation
\begin{equation}\label{eq:elkieseq}
(y+1)^{q-1}\cdot y=\frac{x^q}{(x+1)^{q-1}}.
\end{equation}
This is an optimal tower, which was also studied in detail in
\cite{bezerragarcia}. It is a subtower of the tower defined by (\ref{eq:jnteq}).

In this paper we elaborate further on the ideas of Elkies. We show how
the defining equations for these modular towers can be read of
directly from the modular polynomial, and how this, for higher levels,
leads to recursions of depth 2. With this approach, finding new
towers turns to be an easy task, once the corresponding modular
polynomials are known. To illustrate this, we work out the equations
for the first few cases of Drinfeld modular towers. Unfortunately,
finding the modular polynomials is not an easy task. In fact even the modular polynomial $\Phi_T(X,Y)$ is not known for general $q$. However, in Section \ref{section:two} we find explicit equations for the Drinfeld modular curves $X_0(T^n)$ using a different approach. We
concentrate in this paper on the Drinfeld modular setting. The case of elliptic modular towers can however be treated in exactly the same way.

After having studied towers $(X_0(P^n))_{n\geq 0}$, we investigate the
particular case of the tower $(X_0(T^n)_{n\geq 2}$ more elaborately.
Using expressions from \cite{schweizer}, we work out the defining
equation for this tower (and a variant of it) in unreduced form (over $
\mathbb F_q(T)$). By reducing this modulo various primes other then $T
$, we obtain a whole family (depending on a parameter $\gamma \in
\overline{\mathbb F}_q$) of optimal towers, which contains as particular
cases the towers given by  Equation~(\ref{eq:jnteq}) and Equation~(\ref{eq:elkieseq}). More precisely, we obtain the following theorem (see Remark \ref{rem:optimaltowers1} and Theorem \ref{thm:tower2}):
\begin{theorem}
Let $\mathbb F_q$ be a finite field and $\gamma$ be a nonzero
algebraic element over $\mathbb F_q$. Let $K$ be the quadratic
extension of $\mathbb F_q(\gamma)$ and let the tower $F_0 \subseteq
F_1 \subseteq \ldots$ over $K$, be given by
$F_0=K(x_0)$ and  $F_{n+1}=F_n(x_{n+1})$, with
$$x_{n+1}^q+x_{n+1}=\gamma \cdot \frac{x_n^q}{ x_n^{q-1}+1},$$
for $n\geq 0$ and the tower $E_0 \subseteq E_1 \subseteq \ldots$ over $K
$, be given by
$E_0=K(x_0)$ and  $E_{n+1}=E_n(x_{n+1})$, with
$$(y_{n+1}+1)^{q-1}\cdot y_{n+1}=\gamma^{q-1}\cdot \frac{y_n^q}{(y_n
+1)^{q-1}},$$
for $n\geq 0$.
Then $(F_n)_{n\geq 0}$ and $(E_n)_{n\geq 0}$ are asymptotically
optimal towers over $K$, i.e., attain the Drinfeld--Vladut bound.
\end{theorem}
Note that by taking  $\gamma=1$ one recovers the towers given by
Equation~\ref{eq:jnteq} and Equation~\ref{eq:elkieseq}, respectively.
The optimality of $(E_n)_{n\geq 0}$ follows directly from
\cite{ge2}, whereas the tower  $(F_n)_{n\geq 0}$ requires some
more work.  We identify it as a partial Galois closure of the tower $(E_n)_{n\geq 0}$.
Alternatively, one can attempt to find a more elementary proof of
this, since the towers themselves are given by explicit equations. The
computations of the genera of $E_n$ and $F_n$ can be done in exact
analogy of \cite{bezerragarcia} and \cite{GSJNT} , respectively. Finding many
rational places however is more tricky. Using only explicit methods (a
certain class of recursively defined polynomials and some remarkable
identities they satisfy), we are able to show that over some fixed
finite extension field of $\mathbb F_q(\alpha)$, more precisely the
splitting field of a certain polynomial, there are totally splitting
places. To show that only taking a quadratic extension of $\mathbb F_q(\alpha)$ is sufficient, we use
modular theory. More precisely, we identify the splitting places with
the supersingular points of the corresponding reduction of the curve
$X_0(T)$, which are known to be defined already over a quadratic extension. So in particular the recursively defined polynomials introduced to study the splitting locus can be seen as a certain
analog of the Deuring polynomial to the case of Drinfeld modules.

\section{The Drinfeld modular towers $(X_0(P^n))_{n\ge0}$}\label{section:one}

In this section we will restrict ourselves to the case of Drinfeld modular curves. The case of elliptic modular curves is analogous.
We denote by $\FF$ the field $\F(T)$ and let $N \in \F[T]$ be a monic polynomial. Let $\phi$ be a Drinfeld module of rank two with $j$-invariant $j_0$ and $\phi'$ be an $N$-isogenous Drinfeld module with $j$-invariant $j_1$. The Drinfeld modular polynomial $\Phi_N(X,Y)$ relates these $j$-invariants, more precisely it holds that $\Phi_N(j_0,j_1)=0$. Thinking of $j_0$ as a transcendental element, we can use this equation to define a so called Drinfeld modular curve $X_0(N)$. If we want to emphasize the role of $N$, we will write $j_1=j_1(N)$. It should be noted that $j_0$ is independent of $N$, but it will be convenient to define $j_0(N):=j_0$. The function field $\FF(X_0(N))$ of $X_0(N)$ is therefore given by $\FF(j_0(N),j_1(N))$. Moreover, it is known that

\begin{equation}\label{eq:extdegree}
[\FF(j_0(N),j_1(N)):\FF(j_0(N))]=q^{\deg(N)}\prod_{\substack{A | N \\\ A \ \mathrm{ prime}}}\left(1+\frac{1}{q^{\deg(A)}}\right).
\end{equation}

In principle the work of finding an explicit description of the function field $\FF(X_0(N))$ is done, once the modular polynomial $\Phi_N(X,Y)$ has been computed. However, for general $q$ the Drinfeld modular polynomial is not known explicitly, not even in the case $N=T$. For a given $q$ it can be computed, but this is not always an easy task, since the coefficients of this polynomial tend to get very complicated as the degree of the polynomial $N$ increases. However, following Elkies's ideas (\cite{elkies,elkiesd}) from the modular polynomial $\Phi_P(X,Y)$ for a fixed polynomial $P$, the function fields of the Drinfeld modular curves $X_0(P^n)$ can be described easily in an explicit way. The reason for this is that for polynomials $P,Q \in \F[T]$ a $PQ$-isogeny can be written as the composite of a $P$-isogeny and a $Q$-isogeny, which implies that there is a natural projection from $X_0(PQ)$ to $X_0(P)$ or equivalently an inclusion of function fields $\FF(X_0(P)) \subset \FF(X_0(PQ))$. This implies that function field $\FF(X_0(P^n))$ also contains the function fields $\FF(X_0(P^e))$, for integer satisfying $1\le e \le n$, and hence $j_1(P^e) \in \FF(X_0(P^n))$. Defining $j_e(P):=j_1(P^e)$ for $e \ge 1$, we see that $j_e(P) \in \FF(X_0(P^n))$ for $1 \le e \le n$. Since $j_0$ is independent of $P$, we also have $j_0(P)=j_0(P^n) \in \FF(X_0(P^n))$. Therefore the field $\FF(X_0(P^n))=\FF(j_0(P),j_n(P))=\FF(j_0(P),j_1(P),\dots,j_{n-1}(P),j_n(P))$, is the composite of the fields $\FF(j_e(P),j_{e+1}(P))$ for $e=0,\dots,n-1$. Since $P^{e+1}=PP^e$, any $P^{e+1}$-isogeny can be written as the composite of a $P$-isogeny and a $P^e$-isogeny. This means that $j_e(P)$ and $j_{e+1}(P)$ correspond to $P$-isogenous Drinfeld modules and hence we have $\Phi_P(j_e(P),j_{e+1}(P))=0$ for any $e$ between $0$ and $n-1$. We see that $\FF(X_0(P^n))$ is the composite of $n$ fields isomorphic to $\FF(X_0(P))=\FF(j_0(P),j_1(P))$, the function field of $X_0(P)$. This observation led Elkies to construct a number of recursively defined {\emph{towers}} $(X_0(P^n))_{n\ge 2}$ of modular curves in \cite{elkies,elkiesd}. In \cite{elkies} several models defined over $\Q$ of classical modular curves are given, while in \cite{elkiesd} the reduction mod $T-1$ of the Drinfeld modular tower $X_0(T^n)_{n\ge 2}$ was described.

We consider the function field of $X_0(P^n)$. We have $$\FF(X_0(P^n))=\FF(j_0(P),j_1(P),\dots,j_{n-1}(P),j_n(P)).$$ So we can think of $\FF(X_0(P^n))$ as iteratively obtained from $\FF(j_0(P))$ by adjoining the elements $j_1(P),j_2(P),\dots,j_n(P)$, where $j_{e+1}(P)$ is a root of the polynomial $\Phi_P(j_{e}(P),t) \in \FF(X_0(P^e))[t]$ for $0 \le e < n$. However, except for $j_1(P)$ these polynomials are not irreducible. In fact the extension $\FF(X_0(P^2))/\FF(X_0(P))$ has degree $q^{\deg P}$ by Equation (\ref{eq:extdegree}). This means that the polynomial $\Phi_P(j_1(P),t) \in \FF(j_0(P),j_1(P))[t]$ has a factor $\Psi_P(j_0(P),j_1(P),t)$ of degree $q^{\deg P}$ such that $\Psi(j_0(P),j_1(P),j_2(P))=0$. By clearing denominators if necessary, we can assume that $\Psi_P(j_0(P),j_1(P),t) \in \FF[j_0(P),j_1(P)][t]$. Then clearly the trivariate polynomial $\Psi_P(X,Y,Z) \in \FF[X,Y,Z]$ satisfies $\Psi_P(j_{e-1}(P),j_e(P),j_{e+1}(P))=0$ for all $0<e<n$. The function field $\FF(X_0(P^n))$ can therefore be generated recursively by the equations $\Phi_P(j_0(P),j_1(P))=0$ and $\Psi_P(j_{e-1}(P),j_e(P),j_{e+1}(P))=0$ for $0<e<n$. Note that the depth of the recursion is two in general. We arrive at the following proposition.

\begin{proposition}\label{prop:depth two}
Let $P \in \F[T]$ be a polynomial and $n \ge 0$ an integer. The function field $G_n$ of the Drinfeld modular curve $X_0(P^n)$ is generated by elements $j_0,\dots,j_n$ satisfying:
$$\Phi_P(j_0,j_1)=0,$$ with $\Phi_P(X,Y)$ the Drinfeld modular polynomial corresponding to $P$ and
$$\Psi_P(j_{e-1},j_e,j_{e+1})=0, \makebox{ for } 1\le e <n,$$ with $\Psi_P(X,Y,Z)$ a suitable trivariate polynomial of $Z$-degree $q^{\deg P}$.
Consequently, the tower of function fields $\mathcal{G}:=(G_n)_{n \ge 0}$ can be recursively defined by a recursion of depth two in the following way:
$$G_0:=\FF(j_0),$$ $$G_1:=\FF(j_0,j_1), \makebox{ where } \Phi_P(j_0,j_1)=0$$ and for $n\ge 1$ $$G_{n+1}:=G_{n}(j_{n+1}) \makebox{ where } \Psi_P(j_{n-1},j_n,j_{n+1})=0.$$
\end{proposition}

\begin{remark}\label{rem:casePprime}
The polynomial $\Psi_P(X,Y,Z)$ is easy to describe if $P$ is a prime. In that case $\deg_Y(\Phi_P(X,Y))=q^{\deg P}+1$. Since $\Phi_P(X,Y)$ is a symmetric polynomial, it holds that $$\Phi_P(j_1(P),j_0(P))=\Phi_P(j_0(P),j_1(P))=0.$$ Therefore, the polynomial $\Phi_P(j_1(P),t) \in \FF(X_0(P))[t]$ has the factor $t-j_0(P)$. The factor $\Psi(j_0(P),j_1(P),t)$ can be obtained by dividing $\Phi_P(j_1(P),t)$ by $t-j_0(P)$. Note that in this case automatically $\deg_t \Psi_P(j_0(P),j_1(P),t)=q^{\deg P}$ and $\Psi_P(j_0(P),j_1(P),j_2(P))=0$, as desired.
\end{remark}

By \cite{schweizer3} $X_0(P)$ is rational if and only if $P$ has degree one or two. In that case the tower $(\FF(X_0(P^n)))_{n\ge 1}$ can be generated in a more simple way. Let $e\ge 1$ and let $u_{e-1}(P)$ be a generating element of $\FF(j_{e-1}(P),j_e(P))$ over $\FF$. Then $j_{e-1}(P)=\psi(u_{e-1}(P))$ and $j_e(P)=\phi(u_{e-1}(P))$ for certain rational functions $\psi(t)=\psi_0(t)/\psi_1(t)$ and $\phi(t)=\phi_0(t)/\phi_1(t)$. Here $\psi_0(t)$ and $\psi_1(t)$ (resp. $\phi_0(t)$ and $\psi_1(t)$) denote relatively prime polynomials. Since $\FF(u_{e-1}(P))=\FF(j_{e-1}(P),j_e(P))$, one can generate the function field of $X_0(P^n)$ for $n \ge 1$ by $u_0(P),\dots,u_{n-1}(P)$. These generating elements satisfy the equations $\psi(u_e(P))=\phi(u_{e-1}(P))$ with $1\le e <n$, since $\psi(u_e(P))=j_e(P)=\phi(u_{e-1}(P))$. Similarly as before, one can find generating relations of minimal degree by taking a factor $f_P(u_0(P),t)$ of $\psi_0(t)\phi_1(u_{0}(P))-\psi_1(t)\phi_0(u_{0}(P))$ of degree $q^{\deg P}$ such that $f(u_0(P),u_1(P))=0$. The function field $\FF(X_0(P^n))$ with $n\ge 1$ can then recursively be defined by the equations $f(u_{e-1},u_e)=0$ for $1\le e <n$. We arrive at the following proposition.

\begin{proposition}\label{prop:depth one}
Let $P \in \F[T]$ be a polynomial of degree one or two and $n \ge 0$ an integer. There exists a bivariate polynomial $f_P(X,Y) \in \FF[X,Y]$ of $Y$ degree $q^{\deg P}$ such that the function field $G_n$ of the Drinfeld modular curve $X_0(P^n)$ is generated by elements $u_0,\dots,u_{n-1}$ satisfying:
$$f_P(u_{e-1},u_e)=0, \makebox{ for } 1\le e < n.$$
Consequently, the tower of function fields $\mathcal{G}:=(G_n)_{n \ge 1}$ can be defined by a recursion of depth one:
$$G_1:=\FF(u_0)$$ and for $n\ge 1$ $$G_{n+1}=G_{n}(u_{n+1}) \makebox{ where } f_P(u_n,u_{n+1})=0.$$
\end{proposition}

Finally, if $P$ is a polynomial of degree one, then both $X_0(P)$ and $X_0(P^2)$ are rational. In that case, there exist $u_{e-1}(P),u_e(P)$ as above and $v_{e-1}(P)$ such that $\FF(u_{e-1}(P),u_{e}(P))=\FF(v_{e-1}(P))$ for $e>0$. Similarly as above, there exist rational functions $\psi'(t)$ and $\phi'(t)$ such that
$u_{e-1}(P)=\psi'(v_{e-1}(P))$ and $u_e(P)=\phi'(v_{e-1}(P))$. These rational functions will have degree $q^{\deg P}=q$, since $[\FF(v_{e-1}(P)):\FF(u_{e-1}(P))]=[\FF(v_{e-1}(P)):\FF(u_{e}(P))]=q.$ The function field $\FF(X_0(P^n))$ with $n\ge 2$ can then recursively be defined by the equations $\psi'(v_e(P))=\phi'(v_{e-1}(P))$ for $1\le e <n-1$. The depth of the recursion is one and moreover, the variables can be separated in the defining equations. Since we assume $\deg P=1$, this puts a heavy restriction on the number of possibilities. In fact, without loss of generality we may assume that $P=T$. In the next section we will describe this case in detail, obtaining explicit equations describing the Drinfeld modular tower $\FF(X_0(T^n))_{n\ge 2}$. In the case of classical modular curves, Elkies in \cite{elkies} gave, among others, several similar examples by considering numbers $p$ such that the genus of the classical modular curves $X_0(p)$ and $X_0(p^2)$ is zero.

Given a prime $L \in \F[T]$, we denote by $\FF_L$ the finite field $\F[T]/(L)$. Moreover, we write $\FF_L^{(2)}$ for the quadratic extension of $\FF_L$.  Gekeler showed in \cite{ge2} that the reduction modulo any prime $L \in \F[T]$ not dividing $P$ of the tower $(X_0(P^n))_{n\ge 0}$ gives rise to an asymptotically optimal tower over the constant field $\FF_L^{(2)}$. This means that the tower found in \cite{elkiesd}, being the reduction of $(X_0(T^n))_{n\ge 0}$ modulo $T-1$, is asymptotically optimal over the constant field $\FF_{T-1}^{(2)}=\mathbb{F}_{q^2}.$

Now we will give several examples. Sometimes we do not give all details, since this would fill many pages. Several computations were carried out using the computer algebra package MAGMA \cite{magma}. For example all Drinfeld modular polynomials below were calculated using MAGMA. On occasion, we will perform all calculations sketched above for a reduced version of the tower $(\FF(X_0(P^n)))_{n\ge 0}$, since the resulting formulas are usually much more compact after reduction. In all examples in this section, it is assumed that $q=2$, while $P$ will be a polynomial of degree one or two.

\begin{example}[$P=T, q=2$]
By \cite{schweizer}, the Drinfeld modular polynomial of level $T$ in case $q=2$ is given by
\begin{equation*}
\begin{split}
\Phi_{T}(X,Y)  & = X^3+Y^3+T(T+1)^3(X^2+Y^2)+T^2(T+1)^6(X+Y)\\
 &+T^3(T+1)^9+X^2Y^2+(T+1)^3(T^2+T+1)XY+T(X^2Y+XY^2).
\end{split}
\end{equation*}
The polynomial $\Psi_T(X,Y,Z)$ can readily be found using Remark \ref{rem:casePprime}:
\begin{equation*}
\begin{split}
\Psi_{T}(X,Y,Z)
& =  Z^2 + (X + (Y^2 + TY + T(T+1)^3))Z + X^2 + (Y^2 + TY + T(T+1)^3)X + TY^2\\
& + (T^2+T+1)(T+1)^3Y + T^2(T+1)^6\\
\end{split}
\end{equation*}
Using Proposition \ref{prop:depth two}, we can in principle now describe the tower of function fields of the modular curves $(X_0(T^n))_{n \ge 0}$. However, we can use Proposition \ref{prop:depth one} to find a recursive description of depth one. First we need a uniformizing element $u_0$ of $\FF(j_0,j_1)$.
Using a computer, one finds
$$u_0=\frac{T^3(T^2j_0+T^2+T^4+T^6+1+Tj_1+T^2j_1+Tj_0+j_0j_1)}{(T^3+j_1^2+T^2+j_0+Tj_1+T^3j_0+T^7+T^4j_1+T^6}.$$
Expressing $j_0$ and $j_1$ turns out to give a more compact formula.

$$j_0=\dfrac{(u_0+T)^3}{u_0} \makebox{ and } j_1=\dfrac{(u_0+T^2)^3}{u_0^2}.$$

This means that the variables $u_0$ and $u_1$ satisfy the equation:
$$\dfrac{(u_0+T^2)^3}{u_0^2}=\dfrac{(u_1+T)^3}{u_1}.$$
However, this is not an equation of minimal degree. As explained before Proposition \ref{prop:depth one}, we can find an equation of degree (in this case) two by factoring:
$$(X+T^2)^3Y+(Y+T)^3X^2=(XY+T^3)(X^2+XY^2+XYT+YT^3).$$ We find that $f_T(X,Y)=X^2+XY^2+XYT+YT^3$. This polynomial recursively defines the tower of function fields of the modular curves $(X_0(T^n))_{n \ge 1}$ as in Proposition \ref{prop:depth one}.
According to the discussion following Proposition \ref{prop:depth one} an even more structured description is possible. Since the case for general $q$ is done in the next section, we will not derive that description here.
\end{example}

\begin{example}[$P=T^2+T+1, q=2$]
The Drinfeld modular polynomial of level $T^2+T+1$ is given by
\begin{equation*}
\begin{split}
\Phi_{T^2+T+1}(X,Y)  & = X^5+Y^5 + X^4Y^4 + (T^2 + T + 1)(X^4Y^2+X^2Y^4)\\
& + (T^2 + T + 1)(X^4Y+XY^4) + T^3(T+1)^3(T^2+T+1)(X^4+Y^4)  \\
& + T^2(T+1)^2(T^2+ T + 1)X^3Y^3 \\
& + T(T+1)(T^2+T+1)(T^3+T+1)(T^3+T^2+1)(X^3Y^2+X^2Y^3) \\
& + T^3(T+1)^3(T^2+T+1)(X^3Y+XY^3) + T^6(T+1)^6(T^2+T+1)^2(X^3+Y^3)\\
& + T^5(T+1)^5(T^2+T+1)(T^4+T+1)X^2Y^2\\
& + T^6(T+1)^6(T^2+T+1)(T^4+T+1)(X^2Y+XY^2)\\
& + T^9(T+1)^9(T^2+T+1)^3(X^2+Y^2) + T^{11}(T+1)^{11}X Y
\end{split}
\end{equation*}
As in the previous example one can use Remark \ref{rem:casePprime}, to find the trivariate polynomial $\Psi_{T^2+T+1}(X,Y,Z)$. Finding a uniformizing element $u_0$ of $\F(X_0(T^2+T+1))$ is more tricky and in fact turns out to fill several pages. Below we will state the reduction of $u_0$ modulo $T$ and $T+1$, so the reader can get an impression of its form. Once $u_0$ is found, $j_0$ and $j_1$ can be expressed in terms of it. In this case we find:
$$j_0=\dfrac{(u_0+1)^3(u_0^2+u_0+T^2+T+1)}{u_0} \makebox{ and } j_1=\dfrac{(u_0+T^2+T+1)^3(u_0^2+u_0+T^2+T+1)}{u_0^4}$$
To find the polynomial $f_{T^2+T+1}(X,Y)$, we need to factor the polynomial

\begin{equation*}
\begin{split}(Y^5+(T^2+T+1)Y^3+(T^2+T+1)Y^2+(T^2+T)Y+(T^2+T+1))X^4 + & \\ Y(X^5+(T^2+T)X^4+(T^2+T+1)^2X^3+(T^2+T+1)^3X^2+(T^2+T+1)^4),
\end{split}
\end{equation*}
whose factors are $XY+T^2+T+1$ and
\begin{equation*}
\begin{split}
f_{T^2+T+1}(X,Y)& =Y^4X^3 + (T^2 + T + 1)(Y^3X^2 + Y^2X^3 + (T^2 + T + 1)Y^2X + YX^3 \\ & + (T^2 + T + 1)YX^2 + (T^2 + T + 1)^2Y) + X^4
\end{split}
\end{equation*}
The polynomial $f_{T^2+T+1}(X,Y)$ recursively defines the tower of function fields of the modular curves $(X_0((T^2+T+1)^n))_{n \ge 1}$ as in Proposition \ref{prop:depth one}.

We consider the reduction modulo $T$ or $T+1$ of this tower, which gives an optimal tower over $\mathbb{F}_4$. While a uniformizing element of $\FF(X_0(T^2+T+1))$ was too long to be stated, over $\mathbb{F}_4(X_0(T^2+T+1))$ is given by
$$
u_0:=\dfrac{j_0^4j_1^3 + j_0^4j_1^2 + j_0^4j_1 + j_0^4 + j_0^3j_1^7 + j_0^3j_1^6 + j_0^3j_1^4 + j_0^2j_1^5
    + j_0 j_1^5 + j_0 j_1^4 + j_1^6 + j_1^4}{j_1^8}
$$
Reducing the above found polynomial $f_{T^2+T+1}(X,Y)$ modulo $T$ or $T+1$, we find that the polynomial
$$Y^4X^3 + Y^3X^2 + Y^2X^3 + Y^2X + YX^3+ YX^2 + Y + X^4$$
recursively defines an optimal tower over $\FF_4$.
\end{example}

\begin{example}[$P=T^2+T, q=2$]
In the previous examples, the polynomial $P$ was a prime, but in this example we will consider the composite polynomial $P=T^2+T$. The Drinfeld modular polynomial of level $T^2+T$ has $Y$-degree $9$ by Equation \ref{eq:extdegree}. Using a computer, one finds:
\begin{equation*}
\begin{split}
\Phi_{T^2+T}(X,Y)
& =  X^9 +Y^9+ (X^8Y^4+X^4Y^8) + (T^2 + T + 1)(X^8Y^2+X^2Y^8) \\
&\hspace{-2cm} + (T^2 + T)(X^8Y+XY^8) + (T^6+T^5+T^3+T^2+1)(T^2+T)(X^8+Y^8) + (X^7Y^4+X^4Y^7) \\
&\hspace{-2cm} + (T^2 + T)^3(X^7Y^3+X^3Y^7) + (T^5+T^4+T^3+T+1)(T^5+T^3+T^2+T+1)(T^2+T)^3(X^7+Y^7)\\
&\hspace{-2cm} + (X^6Y^5+X^5Y^6)+ (X^6Y^4+X^4Y^6) + (T^2+T+1)^5(X^6Y^3+X^3Y^6)\\
&\hspace{-2cm} + (T^7+T^6+T^5+T^4+T^2+T+1)(T^7+T^3+T^2+T+1)(T^2+T)(X^6Y^2+X^2Y^6)\\
&\hspace{-2cm} + (T^{14}+T^{13}+T^{11}+T^{10}+T^7+T^5+T^4+T^2+1)(T^2+T)^2(X^6Y+XY^6)\\
&\hspace{-2cm} + (T^4+T+1)(T^2+T+1)(T^2+T)^5(T^8+T^6+T^5+T^4+T^3+T+1)(X^6+Y^6)\\
&\hspace{-2cm} + X^5Y^5 + (T^2+T+1)(T^2+T)^2(X^5Y^4+X^4Y^5) + (T^2 + T)^2(X^5Y^3 +X^3Y^5)\\
&\hspace{-2cm} + (T^9+T^8+T^7+T^5+1)(T^9+T^7+T^6+T^3+T^2+T+1)(X^5Y^2+X^2Y^5)\\
&\hspace{-2cm} + (T^6+T^5+T^2+T+1)(T^6+T^5+1)(T^2+T+1)^3(T^2+T)^2(X^5Y+XY^5)\\
&\hspace{-2cm} + (T^5+T^3+T^2+T+1)(T^5+T^4+T^3+T+1)(T^2+T+1)(T^2+T)^5(X^5+Y^5)\\
&\hspace{-2cm} + (T^{18}+T^{17}+T^{16}+T^{10}+T^9+T^4+T^2+T+1)(T^2+T+1)^2(T^2+T)(X^4Y^2+X^2Y^4)\\
&\hspace{-2cm} + (T^2+T+1)^2(T^2+T)^7(X^4Y+XY^4)+ (T^2+T)^8(T^6+T^5+T^3+T^2+1)(X^4+Y^4)\\
&\hspace{-2cm} + (T^{10}+T^9+T^8+T^6+T^5+T+1)(T^2+T+1)^3X^3Y^3\\
&\hspace{-2cm} + (T^8+T^7+T^2+T+1)(T^8+T^7+T^6+T^5+T^4+T^3+1)(T^2+T+1)(T^2+T)^2(X^3Y^2+X^2Y^3)\\
&\hspace{-2cm} + (T^2+T+1)(T^2+T)^4(T^{10}+T^9+T^8+T^3+T^2+T+1)(X^3Y+XY^3)\\
&\hspace{-2cm} + (T^4+T+1)(T^3+T+1)(T^3+T^2+1)(T^2+T+1)^3(T^2+T)^3X^2Y^2\\
&\hspace{-2cm} + (T^2+T)^{10}(X^2Y+XY^2) + (T^2+T)^{10}(X^2+Y^2)+ (T^4+T+1)(T^2+T)^7(X^3+Y^3)\\
&\hspace{-2cm} + (T^3+T+1)(T^3+T^2+1)(T^2+T)^6XY + (T^2+T+1)(T^2+T)^8(X+Y) + (T^2+T)^9\\
\end{split}
\end{equation*}
Finding a uniformizing element $u_0$ of $\F(X_0(T^2+T))$ and expressing $j_0$ and $j_1$ in it, we find
$$j_0=\dfrac{(u_0^3 + (T^2 + T)u_0 +(T^2 + T))^3}{u_0(u_0+T)^2(u_0+T+1)^2} \makebox{ and } j_1=\dfrac{(u_0^3 + (T^2 + T)u_0^2 +(T^2 + T)^2)^3}{u_0^4(u_0+T)^2(u_0+T+1)^2}$$
To find $f_{T^2+T}(X,Y)$, we need to factor a bivariate polynomial of $Y$-degree $9$. Note that Remark \ref{rem:casePprime} does not apply, though it still predicts the existence of one factor of $Y$-degree one. The factors turn out to be
$$ XY + T^2 + T,$$
$$ Y^2X^2 + TY^2X + (T^2 + T)YX + (T^3 + T^2)Y + T^2X^2 + T^4 + T^2,$$
$$ Y^2X^2 + (T + 1)Y^2X + (T^2 + T)YX + (T^3 + T)Y + (T^2 + 1)X^2 + T^4+ T^2,$$
and
\begin{equation*}
\begin{split}
Y^4X^3 + Y^4X^2 + (T^2 + T)Y^4X + (T^2 + T)Y^3X^2 + (T^2 + T)Y^3X +(T^4 + T^2)Y^3 + (T^2 + T + 1)Y^2X^3 &\\
+ (T^4 + T^2)Y^2X + (T^4 + T^2)Y^2 + (T^2 + T)YX^3 + (T^4 + T)YX^2 + (T^6 + T^5 + T^4 + T^3)Y + X^4.
\end{split}
\end{equation*}
The last factor is $f_{T^2+T}(X,Y)$, since it is the only factor of $Y$-degree $4$.
Considering reduction modulo $T^2+T+1$, we see that the polynomial
$$Y^4X^3 + Y^4X^2 + Y^4X + Y^3X^2 + Y^3X +Y^3 + Y^2X + Y^2 + YX^3 + Y + X^4$$ recursively defines an optimal tower over $\mathbb{F}_{16}$.
\end{example}

\section{The Drinfeld modular tower $(X_0(T^n))_n$}\label{section:two}

For general $q$ the Drinfeld modular polynomial is not known explicitly, not even in the case $P=T$. Nonetheless, for $P=T$ Schweizer \cite{schweizer} found an explicit description of the relation between $j_0$ and $j_1$ as well as a uniformizing parameter of the function field of $X_0(T)$. More precisely, he showed that $\FF(X_0(T))$ is given by $\FF(u_0)$, where
\begin{equation}\label{eq:jandjprime}
j_0=\dfrac{(u_0+T)^{q+1}}{u_0}  \makebox{ and } j_1=\dfrac{(u_0+T^q)^{q+1}}{u_0^q}.
\end{equation}
We will use this description to find an explicit description of the function fields $\FF(X_0(T^n))$ for any $n$.

There exist a model of the curve $X_0(T^n)$ whose reduction modulo any prime element $L \in \F[T]$ different from $T$ gives rise to a curve defined over the finite field $\FF_L$. We will denote this reduced curve by $X_0(T^n)/\FF_L$. Reduction modulo such a prime $L$ gives rise to an optimal tower $(X_0(T^n)/\FF_L^{(2)} )_{n\ge 0}$ over the finite field $\FF_L^{(2)}$. In \cite{elkies} several models defined over $\Q$ of classical modular curves are given, while in \cite{elkiesd} the reduction mod $T-1$ of the Drinfeld modular tower $X_0(T^n)_{n\ge 2}$ was described. We will use Schweizer's description of $X_0(T)$ to obtain explicit equations describing the modular tower $\FF(X_0(T^n))_{n\ge 2}$ in \emph{unreduced} form, i.e.~over $\FF$. For future usage, we make the following definition.

\begin{definition}\label{def:ff}
Let $n\ge 0$ be an integer. Then we define $E_n=\FF(X_0(T^{n+2}))$. Furthermore for a prime $L \in \F[T]$ different from $T$, we denote by $E_n^{(L)}$ the function field of the curve $X_0(T^{n+2})/\FF_L^{(2)}$.
\end{definition}

Schweizer's description of $\FF(X_0(T))$ and Proposition \ref{prop:depth one}, enables one to identify the function field $E_{0}=\FF(X_0(T^2))$ with $\FF(u_0,u_1)$ where $u_0$ and $u_1$ satisfy the equation
\begin{equation}\label{eq:schweizer}
\dfrac{(u_0+T^q)^{q+1}}{u_0^q}=\dfrac{(u_1+T)^{q+1}}{u_1}.
\end{equation}
As before, the variable $u_0$, resp. $u_1$ denotes the generator of the field $\FF(j_0(T),j_1(T))$, resp. the field $\FF(j_1(T),j_2(T))$. Since the genus of $X_0(T^2)$ is zero, our first task is to find a generator of its function field. From the discussion before Proposition \ref{prop:depth one} we see that Equation (\ref{eq:schweizer}) is not the equation of lowest degree connecting $u_0$ and $u_1$. Since
$$\dfrac{(u_0+T^q)^{q+1}}{u_0^q}-\dfrac{(u_1+T)^{q+1}}{u_1}=\left(u_0-\frac{T^{q+1}}{u_1}\right)\left(1+\frac{T^{q^2}}{u_0^q}
-\left(u_1-\frac{T^{q+1}}{u_0}\right)^{q-1}\left(\frac{u_1}{u_0}+\frac{T}{u_0}\right) \right),$$
we conclude that $u_0$ and $u_1$ satisfy the equation
$$1+\frac{T^{q^2}}{u_0^q}-\left(u_1-\frac{T^{q+1}}{u_0}\right)^{q-1}\left(\frac{u_1}{u_0}+\frac{T}{u_0}\right)=0,$$
which can be rewritten as
$$\frac{u_0+T^q}{u_1+T}=\left(\frac{u_0u_1-T^{q+1}}{u_0+T^q}\right)^{q-1}.$$

Hence the element $v_0=(u_0u_1-T^{q+1})/(u_0+T^q)$ is a generator of the function field $E_0=\FF(X_0(T^2))$, and in fact we find
$$u_0=v_0^{q-1}(v_0+T) \makebox{ and } u_1=\frac{(v_0+T)^q}{v_0^{q-1}}.$$
It will be convenient to use a slightly different generator, namely $\xi_0=-(v_0+T)/T$. In terms of $\xi_0$, we find
\begin{equation}\label{eq:correspondence}
u_0=-T^q(\xi_0+1)^{q-1}\xi_0 \makebox{ and } u_1=-T\frac{\xi_0^q}{(\xi_0+1)^{q-1}}.
\end{equation}
The function field $E_n$, $n\ge 0$ can be generated by elements $\xi_0,\dots,\xi_{n}$ satisfying the equations $$-T^q(\xi_{i}+1)^{q-1}\xi_i=-T\frac{\xi_{i-1}^q}{(\xi_{i-1}+1)^{q-1}} \quad (1\le i \le n)$$ which simplifies to the equations
\begin{equation}\label{eq:unreducedElkies}
(\xi_{i}+1)^{q-1}\xi_i=\frac{\xi_{i-1}^q}{T^{q-1}(\xi_{i-1}+1)^{q-1}} \quad (1\le i \le n).
\end{equation}
In fact we have shown the following:
\begin{theorem}
Denote for $n\ge 0$ by $E_{n+2}$ the function field of the Drinfeld modular curve $X_0(T^{n+2})$ over $\FF$. Then the tower of function fields $(E_n)_{n\ge 0}$ can be recursively given as follows:
$$E_0=\FF(\xi_0) \makebox{ and }  E_{n+1}=E_n(\xi_{n+1}),$$ with
$$(\xi_{n+1}+1)^{q-1}\cdot \xi_{n+1}=T^{q-1}\cdot \frac{\xi_n^q}{(\xi_n
+1)^{q-1}},$$
for $n\geq 0$.
\end{theorem}
Reducing modulo $T-1$, i.e.~putting $T=1$, we recover the equation found by Elkies in \cite{elkiesd}. Alternatively, still reducing modulo $T-1$, by making the change of variables $$\zeta_i=\frac{1}{\xi_i+1},$$ we recover the tower in \cite{bezerragarcia}.

Next we recall some facts from the literature on the tower $(X_0(T^n))_{n\ge 2}$ that we will need later on.

\begin{proposition}\label{prop:facts}
The genus of the function field $E_n$ is given by
$$g(E_n)=\dfrac{q^{n+2}+q^{n+1}-(q+1)(2+q+q^{2+\lfloor n/2 \rfloor}+q^{\lfloor (n-1)/2 \rfloor})}{q^2-1}.$$ This genus formula also holds after reduction modulo a prime different from $T$.
\end{proposition}
\begin{proof}
This genus formula can be found in \cite{ge1,schweizer2}. Alternatively, one can adapt the explicit techniques used in \cite{bezerragarcia}.
\end{proof}

\begin{remark}\label{rem:optimaltowers1}
Let $L \in \F[T]$ be a prime different from $T$.  We denote by $\mathcal E^{(L)}$ the tower $(E_n^{(L)})_{n\ge 0}$. We denote
$\alpha\equiv T \pmod{L}$ and interpret it as an element of $\mathbb{F}_L$. Then Equation (\ref{eq:unreducedElkies}) implies that $\mathcal E^{(L)}$ is
recursively defined by
$E_0^{(L)}=\mathbb{F}_{L}^{(2)}(\xi_0)$ and $E_{n+1}^{(L)}=E_{n}^{(L)}(\xi_{n+1})$, where
$$(\xi_{n+1}+1)^{q-1}\xi_{n+1}=\frac{\xi_{n}^q}{\alpha^{q-1}(\xi_{n}+1)^{q-1}}.$$
\end{remark}

\section{The splitting locus of the Drinfeld modular tower $\mathcal{E}^{(L)}$}\label{section:three}

Since we have found an explicit description of the Drinfeld modular tower $\mathcal E^{(L)}$, we can obtain explicit information about its splitting locus. After that, using the modular interpretation of this tower, we will glean some information about supersingular Drinfeld modules in the next section. We will start by investigating a sequence of polynomials that turn out to be related to the splitting locus of $\mathcal E^{(L)}$.

\begin{definition}\label{def:polsplit}
For $i\ge -1$, we recursively define $p_i(s)\in\FF[s]$ by

$$p_{-1}(s)=0  \makebox{, } p_0(s)=1$$
and
\begin{equation}\label{eq:recursion}
p_{i+1}(s)=(s^{q^i}-1)p_i(s)+\frac{T^{q^i-1}-1}{T^{q^i-1}}s^{q^i}p_{i-1}(s).
\end{equation}
\end{definition}
For example we have $$p_1(s)=s-1, \ p_2(s)=s^{q+1}-s^{q}/T^{q-1}-s+1,$$ and $$p_3(s)=s^{q^2+q+1}-s^{q^2+q}/T^{q-1}-s^{q^2+1}/T^{q^2-1}+s^{q^2}/T^{q^2-1}-s^{q+1}+s^{q}/T^{q-1}+s-1.$$
Reducing $p_{i}(s)$ modulo a prime element $L \in \F[T]$, gives rise to a polynomial with coefficients in the finite field $\FF_L$. We will denote this polynomial by $p_i^{(L)}(s)$.

\begin{proposition}\label{prop:simple}
Let $d \ge 1$ be an integer and $L \in \F[T]$ be a prime different from $T$ of degree $d$. All roots of the polynomial $p_d^{(L)}(s)$ are simple. Moreover, $0$ is not a root of $p_d^{(L)}(s)$.
\end{proposition}
\begin{proof}
It is easily seen by induction that $p_i(0)=(-1)^{i}$ for $i \ge 0$. Therefore $0$ is not a root of $p_d^{(L)}(t)$.

We denote by $p_i'(s)$ the derivative of $p_i(s)$ with respect to $s$. By taking the derivative on both sides of the equality sign in Equation (\ref{eq:recursion}), we see that the sequence $p'=(p_0'(s), p_1'(s),p_2'(s),\dots)$ satisfies the same recursion as the sequence $p=(p_0(s), p_1(s),p_2(s),\dots)$. The same holds for any linear combination $\lambda p'+\mu p$ of these sequences.

For convenience, we write $\alpha \equiv T \pmod{L} \in \mathbb{F}_{L}$. Now suppose that $\rho$ is a root of $p_i^{(L)}(s)$ in the algebraic closure of $\F$ of multiplicity greater than one. Since then $\rho$ is a common root of $p_d^{(L)}(s)$ and $\left(p_{d}^{(L)}\right)'(s)$, we can choose $\lambda$ and $\mu$, not both zero, such that $\lambda \left(p_{d-1}^{(L)}\right)'(\rho)+\mu p_{d-1}^{(L)}(\rho)=0$ and $\lambda \left(p_{d}^{(L)}\right)'(\rho)+\mu p_{d}^{(L)}(\rho)=0$. However, using the reduction modulo $L$ of Equation (\ref{eq:recursion}), this implies that $\lambda \left(p_{i}^{(L)}\right)'(\rho)+\mu p_{i}^{(L)}(\rho)=0$ for any $0 \le i \le d$. Here it is essential that $\alpha^{q^i} \neq 1$ for $0 \le i <d$, which holds, since $\alpha$ is a root of the irreducible polynomial $L$ of degree $d$.

For $i=0$, we find $0=\lambda \left(p_0^{(L)}\right)'(\rho)+\mu p_{0}^{(L)}(\rho)=\mu$, while for $i=1$, we find $0=\lambda \left(p_1^{(L)}\right)'(\rho)+\mu p_{1}^{(L)}(\rho)=\lambda+\mu(\rho-1)$. This implies that $\lambda=\mu=0$, a contradiction.
\end{proof}

It turns out that the polynomials $p_i(s)$ also can be defined by a recursion of depth one. For the sake of completeness, we state this recursion:

\begin{lemma}\label{lem:recdepthone}
Let $i\ge 1$ be an integer. The polynomials from Definition \ref{def:polsplit} satisfy:
\begin{equation}\label{eq:recdepthone}
p_i(s)=(-1)^{i+1}s^{\frac{q^i-1}{q-1}}p_{i-1}\left(\frac{1}{T^{q-1}s}\right)-p_{i-1}(s)
\end{equation}
and
\begin{equation}\label{eq:recdepthone2}
p_i\left(\frac{1}{T^{q-1}s}\right)=(-1)^{i+1}(sT^{q-1})^{-\frac{q^i-1}{q-1}}p_{i-1}\left(s\right)-p_{i-1}\left(\frac{1}{T^{q-1}s}\right)
\end{equation}
\end{lemma}
\begin{proof}
The second equation follows directly from the first by changing the variable $s$ to $1/(T^{q-1}s)$. We prove the first equation with induction. For $i=1$ the equation follows by direct computation. Now suppose that Equations (\ref{eq:recdepthone}) and (\ref{eq:recdepthone2}) hold for a certain $i \ge 1$. Then using Equation (\ref{eq:recursion}) and the induction hypothesis, we find that
\begin{multline*}
\begin{split}
p_{i+1}(s)+p_i(s) & =s^{q^i}p_i(s)+\frac{T^{q^i-1}-1}{T^{q^i-1}}s^{q^i}p_{i-1}(s)\\
           & =s^{q^i}\left( (-1)^{i+1}s^{\frac{q^i-1}{q-1}}p_{i-1}\left(\frac{1}{T^{q-1}s}\right)-p_{i-1}(s) \right)+\frac{T^{q^i-1}-1}{T^{q^i-1}}s^{q^i}p_{i-1}(s)\\
           & = (-1)^{i+1}s^{\frac{q^{i+1}-1}{q-1}}p_{i-1}\left(\frac{1}{T^{q-1}s}\right)-\frac{s^{q^i}}{T^{q^i-1}}p_{i-1}(s) = (-1)^{i+2}s^{\frac{q^{i+1}-1}{q-1}}p_i\left(\frac{1}{T^{q-1}s}\right).
\end{split}
\end{multline*}
\end{proof}

Now we return to our main task: to show the connection between the polynomials $p_i(s)$ and the splitting locus of $\mathcal E^{(L)}$. The following theorem contains a key identity:

\begin{theorem}\label{thm:keyresult}
Let $i\ge 0$ be an integer. Then
\begin{equation*}
p_i(s(s+1)^{q-1})-\left(T(s+1)\right)^{q^i-1}p_i\left(\frac{s^q}{(T(s+1))^{q-1}}\right)
  =(T^{q^i-1}-1)\left(s+1\right)^{q^i-1} p_{i-1}\left(\frac{s^q}{(T(s+1))^{q-1}}\right).
\end{equation*}
\end{theorem}
\begin{proof}
We prove the theorem by induction on $i$. For $i=0$ and $i=1$, the theorem is trivial.


Now suppose that $i \ge 1$ and that the theorem is true for $i$ and $i-1$. Then
\begin{multline}\label{eq:intermediate}
\begin{split}
p_{i+1}(s(s+1)^{q-1})
& = ((s(s+1)^{q-1})^{q^i}-1 )p_i(s(s+1)^{q-1})+\frac{T^{q^i-1}-1}{T^{q^i-1}}(s(s+1)^{q-1})^{q^i}p_{i-1}(s(s+1)^{q-1}) \\
& = f_i \, p_{i}\left(\frac{s^q}{(T(s+1))^{q-1}}\right)+f_{i-1}  \, p_{i-1}\left(\frac{s^q}{(T(s+1))^{q-1}}\right)+f_{i-2}  \, p_{i-2}\left(\frac{s^q}{(T(s+1))^{q-1}}\right),
\end{split}
\end{multline}
with
\begin{multline*}
\begin{split}
f_i & =\left( (s(s+1)^{q-1})^{q^i}-1 \right)\left(T (s+1)\right)^{q^i-1} \\
f_{i-1} & =( (s(s+1)^{q-1})^{q^i}-1 )(T^{q^i-1}-1)(s+1)^{q^i-1} + \frac{T^{q^i-1}-1}{T^{q^i-1}}(s(s+1)^{q-1})^{q^i}(T(s+1))^{q^{i-1}-1} \\
f_{i-2} & =\frac{T^{q^i-1}-1}{T^{q^i-1}}(s(s+1)^{q-1})^{q^i}(T^{q^{i-1}-1}-1)(s+1)^{q^{i-1}-1}.
\end{split}
\end{multline*}
In the first equality we used Equation (\ref{eq:recursion}), in the second equality the induction hypothesis. Using Equation (\ref{eq:recursion}), with the variable $s$ replaced by $s^q/(T(s+1))^{q-1}$, to express $p_{i-2}$ in $p_{i-1}$ and $p_i$, we can rewrite Equation (\ref{eq:intermediate}) as
\begin{multline}\label{eq:intermediate2}
\begin{split}
p_{i+1}(s(s+1)^{q-1}) & = \left(f_i+(T^{q^i-1}-1)(s+1)^{q^{i+1}-1}\right)p_i\left(\frac{s^q}{(T(s+1))^{q-1}}\right)\\
&+(s+1)^{q^i-1}s^{q^{i+1}}(T^{q^i-1}-1)p_{i-1}\left(\frac{s^q}{(T(s+1))^{q-1}}\right)
\end{split}
\end{multline}
Using Equation (\ref{eq:recursion}) to express $p_{i-1}$ in terms of $p_{i}$ and $p_{i+1}$ on the right hand side of Equation (\ref{eq:intermediate2}), the theorem follows.
\end{proof}

We have the following consequence.

\begin{corollary}\label{cor:splitting}
Let $L \in \F[T]$ be a prime different from $T$ of degree $d\ge 1$. Let $K$ be compositum of $\mathbb{F}_{L}^{(2)}$ and the splitting field of the polynomial $p_d^{(L)}(s(s+1)^{q-1})$. Let $P$ be a place of $KE_0^{(L)}$ that is a zero of $p_d^{(L)}(\xi_0(\xi_0+1)^{q-1})$. Then the place $P$ splits in the tower $K\mathcal{E}^{(L)}$. Moreover, $$\frac{N(KE_n^{(L)})}{[E_n^{(L)}:E_0^{(L)}]}\ge q\frac{q^d-1}{q-1}.$$
\end{corollary}
\begin{proof}
Define $$S:=\{ a \in K \,|\, p_d^{(L)}(a(a+1)^{q-1}) = 0 \}.$$ Since $p_d(0)\neq 0$, we have $a(a+1)^{q-1} \neq 0$ for any $a\in S$. From Proposition \ref{prop:genusunreducedjnt} we see therefore that $S$ is disjoint from the ramification locus of the tower $K\mathcal{E}^{(L)}$. As before, we denote by $\alpha\in \FF_L$ the reduction of $T$ modulo $L$. Moreover, for any $a\in S$, we denote by $P_a$ the zero of $\xi_0-a$. For an $a\in S$, the minimal polynomial of $\xi_1$ over $E_0^{(L)}$ modulo $P_a$ is given by $\tau(\tau+1)^{q-1}-\xi_0^q/(\alpha(\xi_0+1))^{q-1}$. Reducing at $P_a$, we obtain $\tau(\tau+1)^{q-1}-a^q/(T(a+1))^{q-1}$. We will show that this polynomial has $q$ distinct roots in $S$. It suffices to show this, since by iterating the argument, we then can conclude that $P_a$ splits completely in the tower.

Setting $s=\xi_0$ in Theorem \ref{thm:keyresult}, we find (after reducing modulo $L$ and $P_a$) that $$p_d^{(L)}(a(a+1)^{q-1}) = \left(\alpha(a+1)\right)^{q^d-1}p_d^{(L)}\left(\frac{a^q}{(\alpha(a+1))^{q-1}}\right).$$ Since $a(a+1)^{q-1} \neq 0$ and since $a \in S$, we obtain that $a^q/(\alpha(1+a))^{q-1}$ is a root of $p_d^{(L)}(s)$. By definition of $S$, we therefore conclude that there are $q$ possibilities for $b \in S$ such that $b(b+1)^{q-1}=a^q/(\alpha(1+a))^{q-1}$.

Now we prove the last statement. By Proposition \ref{prop:simple}, all roots of $p_d^{(L)}(s)$ are simple. Since by Proposition \ref{prop:simple}, $0$ is not a root of $p_d^{(L)}(s)$, we see that all roots of the polynomial $p_d^{(L)}(s(s+1)^{q-1})$ are simple as well. The statement now follows directly, since $\deg p_d(s(s+1)^{q-1})=q(q^d-1)/(q-1).$
\end{proof}

In the following section, by identifying the roots of $p_d^{(L)}(s(s+1)^{q-1})$ with supersingular Drinfeld modules in a particular family, we will show that $K=\mathbb{F}_{L}^{(2)}$, i.e. the polynomial $p_d^{(L)}(s(s+1)^{q-1})$ has all its roots in $\mathbb{F}_{L}^{(2)}$. Even more is true: we will show in the last section that any root of the polynomial $p_d^{(L)}(s(s+1)^{q-1})$ is a $(q-1)$-st power in $\mathbb{F}_{L}^{(2)}$. Using Theorem \ref{thm:supersingular} from the next section we can obtain the following optimality result:

\begin{corollary}\label{cor:splitting2}
Let $L \in \F[T]$ be a prime different from $T$ of degree $d$. All zeroes of the function $p_d^{(L)}(\xi_0(\xi_0+1)^{q-1})$ in $E_0^{(L)}$ are rational. The tower $\mathcal E^{(L)}=(E_n^{(L)})_{n\ge 0}$ has limit ${\mathcal E}^{(L)}=q^d-1$ and hence it is optimal.
\end{corollary}
\begin{proof}
This follows from Proposition \ref{prop:facts}, Corollary \ref{cor:splitting} and Theorem \ref{thm:supersingular}.
\end{proof}

Alternatively, one could use the results in \cite{ge2} to obtain the above corollary, since $\mathcal E$ is the Drinfeld modular tower $(\F(X_0(T^{n+2})))_{n\ge0}$. However, we have found an explicit description of the tower $\mathcal E$ as well as its splitting locus as we will see.

\section{The $u$-line.}\label{section:u}

In this section we will consider the family of Drinfeld modules $\phi$ over $\F(T)$ of type:
\begin{equation}\label{eq:type}
\phi_T=u\tau^2+(u+T)\tau+T.
\end{equation}
Note that the $j$-invariant of $\phi$ is given as $(u+T)^{q+1}/u$. This is the same expression as for $j_0$ (in terms of $u_0$) in Equation (\ref{eq:jandjprime}). Let $L \neq T$ be any irreducible element of $\F[T]$ of degree $d$. Since we have seen that the function field of $X_0(T)$ is generated by $u$, the supersingular Drinfeld modules over $\FF_L$ of the above type correspond to the supersingular points on the Drinfeld modular curve $X_0(T)/\FF_{L}$ (these are by definition the points of $X_0(T)/\FF_L$ lying above supersingular points of $X(1)/\FF_L$). According to \cite{ge2} their number is given by $m_d:=(q^d-1)/(q-1)$, while the number of supersingular points of $X_0(T^n)/\FF_L$ is given by $q^{n-1}m_d$. Since the degree of the covering $X_0(T^n) \to X_0(T)$ is $q^{n-1}$, we can conclude that all supersingular points of $X_0(T)/\FF_L$ are unramified. Using the fact that all supersingular points of $X_0(T^n)/\FF_L$ are $\FF_L^{(2)}$-rational, we conclude that they split in the covering $X_0(T^n)/\FF_L^{(2)} \to X_0(T)/\FF_L^{(2)}$. On the other hand in Corollary \ref{cor:splitting}, we have exhibited $\deg(p_d(s))=m_d$ points defined over a suitable finite extension $K$ of $\FF_L^{(2)}$, which split in $X_0(T^n)/K \to X_0(T)/K$. We will now show that these sets of splitting points coincide and hence that $K=\FF_L^{(2)}$.

\begin{theorem}\label{thm:supersingular}
Supersingular Drinfeld modules over $\FF_L$ of type (\ref{eq:type}) are in one to one correspondence with supersingular points of $X_0(T)/\FF_L$. Moreover, they correspond to values of $a \in \FF_L^{(2)}$ satisfying $p_d^{(L)}(a) = 0$. In particular, all roots of $p_d^{(L)}(s) $ are in $\FF_L^{(2)}$.
\end{theorem}
\begin{proof}
It is sufficient to show that the two above mentioned sets of points (each of size $m_d$) of $X_0(T)/\FF_L$ are identical. Since they have the same cardinality $m_d$, it it enough to show that each root of $p_d^{(L)}(s)$ is a supersingular point. Assume there exists a point $Q$ with $u(Q)$ a root of $p_d^{(L)}(s)$, which is not supersingular. Since the roots of $p_d^{(L)}(s)$ split in the tower $(X_0(T^n)/K)_{n\ge 1}$ for a suitable finite extension $K$ of $\FF_L^{(2)}$, such a point $Q$ would give an asymptotically nontrivial number of $K$-rational points at each level (compared to the degree of the covering). However there are $q^{n-1}m_d$  supersingular points of $X_0(T^n)/\FF_L$, each $\FF_L^{(2)}$-rational. Together with the genus formula in Proposition \ref{prop:facts}, we see that the tower $(X_0(T^n)/\FF_L^{(2)})_{n\ge 1}$ is optimal. By the generalisation of the Drinfeld--Vladut bound by Tsfasman--Vladut given in \cite{genDV} such a point $Q$ cannot exist.
\end{proof}

By Equation \ref{eq:correspondence}, the correspondence in the above theorem can be given explicitly as follows: A supersingular value of the uniformizer $u$ corresponds to the root $-u/T^q \pmod{L}$ of $p_d^{(L)}(s)$. Considering the splitting points in the covering $X_0(T^2)/\FF_L^{(2)} \to X_0(T)/\FF_L^{(2)}$ we immediately obtain the following:

\begin{corollary}
The roots of the polynomial $p_d^{(L)}(s(s+1)^{q-1})$ are in $\FF_L^{(2)}$ and correspond to supersingular points of $X_0(T^2)/\FF_L$.
\end{corollary}




The polynomials $p_d(s)$ can be seen as analogues of the Deuring polynomial for Drinfeld modules of type (\ref{eq:type}). Remarkably, the recursion in Definition \ref{def:polsplit} relates the Deuring polynomials associated with different primes. Summing up we have the following:

\begin{corollary}
As in Definition \ref{def:polsplit}, let $p_{-1}(s)=0$, $p_0(s)=1$ and
$$p_{i+1}(s)=(s^{q^i}-1)p_i(s)+\frac{T^{q^i-1}-1}{T^{q^i-1}}s^{q^i}p_{i-1}(s).$$
Let $L\in \FF[T]$ be a prime element of degree $d\ge 1$. The Drinfeld module given by $\phi_T=u\tau^2+(u+T)\tau+T$ is supersingular for the prime $L$ if and only if $p_d(-u/T^q) \equiv 0 \pmod{L}$.
\end{corollary}

Note that in \cite{ge3} Gekeler gives and studies an analogue of the Deuring polynomial for the family of Drinfeld modules given by $\phi_T=\tau^2+\lambda\tau+T$. As the $u$-line above, the $\lambda$-line in \cite{ge3} is a degree $q+1$ covering of the $j$-line given however by the formula $j=\lambda^{q+1}$.

\section{A supertower of $(X_0(T^n))_{n \ge 2}$}\label{section:four}

In the previous section we found explicit equations for the function fields in the modular tower $(X_0(T^n))_{n\ge 2}$. Now we will describe a tower of function fields (with constant field $\FF$) which turns out to be a supertower of the previous one and show that after reduction modulo a prime $\ell$ different from $T$, one obtains an optimal tower. The optimal tower found by Garcia and Stichtenoth in \cite{GSJNT}, will turn out to be a special case of this construction.

\begin{definition}\label{def:unreducedGSjnt}
Define the tower ${\mathcal F}=(F_n)_{n\ge 0}$ of function fields recursively by
$F_0=\FF(x_0)$ and $F_i=F_{i-1}(x_i)$, where
\begin{equation}\label{eq:unreducedGSjnt}
x_{i}^q+x_i=\frac{x_{i-1}^q}{T(x_{i-1}^{q-1}+1)}.
\end{equation}
\end{definition}

By raising to the power $q-1$, it is easily seen that Equation (\ref{eq:unreducedGSjnt}) is connected to Equation (\ref{eq:unreducedElkies}) by the relation $\xi_j=x_j^{q-1}$ for all $j$. Since $$x_i=\frac{x_{i-1}^q}{T(x_{i-1}^{q-1})(\xi_i+1)},$$ this implies that the extension $F_n/E_n$ function fields is a Kummer extension of degree $q-1$. In fact the tower $\mathcal F$ can be obtained by taking the composite over $\FF(\xi_0)$ of the tower $\mathcal E$ and the function field $\FF(x_0)$. The tower $\mathcal F$ is therefore a supertower of $\mathcal E$. We will start the investigation of $\mathcal F$ by determining its ramification locus.

\begin{proposition}\label{prop:genusunreducedjnt}
The ramification locus of the tower $\mathcal F$ is equal to the set consisting of the pole and the zeroes of the function $x_0^q+x_0$. The genus of $F_n$ satisfies:
$$\frac{g(F_n)-1}{[F_n:F_0]} \le q.$$
\end{proposition}
\begin{proof}
The only places of $F_0$ that ramify in the extension $F_1/F_0$ are the poles of $x_0^q/(T(x_0^{q-1}+1))$. This means that only $P_\infty$ and $P_\alpha$ where $\alpha^{q-1}+1=0$ ramify in $F_1/F_0$. On the other hand, let $Q$ be a place of $F_n$ which is ramified over $F_0$. Let $P=Q\cap F_0$ and $Q_i=Q_n\cap F_i$. We see that there exists $1\le i \le n$ such that $e(Q_i|Q_{i-1})>1$. This means that either $x_{i-1}$ has a pole in $Q_{i-1}$ or that $x_{i-1} \pmod{Q_{i-1}} \in  \{a \in \mathbb{F}_{q^2} | a^{q-1}+1=0 \}$. Inspecting the defining equations, we see that this implies that either $x_0$ has a pole at $P$ or that $P$ is a zero of $x_0-a$ with $a^q+a=0$.
From the fact that all extensions in the tower $\mathcal F$ are 2-bounded, a very similar reasoning as in \cite{GSMMJ} gives the result.
\end{proof}

\begin{remark}
Performing similar computations as in \cite{GSJNT} one can determine the exact genus of $F_n$. Alternatively, one could use the Riemann--Hurwitz formula and use that $F_n$ is a Kummer covering of the function field of $X_0(T^{n+2})$. It turns out the genus of $F_n$ is given as follows:
$$g(F_n)=
\left\{
\begin{array}{ll}
(q^{(n+1)/2}-1)^2          & \makebox{if $n$ is odd}\\
(q^{(n+2)/2}-1)(q^{n/2}-1) & \makebox{if $n$ is even}
\end{array}
\right.
$$
This is the same as the genera found in \cite{GSJNT}, which is as it should be, since the tower in \cite{GSJNT} is obtained by reducing the tower $\mathcal F$ at the good prime $T-1$.
\end{remark}

To obtain optimal towers, we will again reduce modulo prime elements of $\F[T]$. In fact, we will consider the following towers:

\begin{definition}
Let $L \in \F[T]$ be a prime different from $T$. As before, we write $\alpha \equiv T \pmod{L}$ and interpret it as an element of $\mathbb{F}_{L}$. We denote by $\mathcal{F}^{(L)}=(F_n^{(L)})_{n\ge 0}$ the tower recursively defined by
$F_0^{(L)}=\mathbb{F}_{L}^{(2)}(x_0)$ and $F_i=F_{i-1}(x_i)$, where
\begin{equation}\label{eq:reducedGSjnt}
x_{i}^q+x_i=\frac{x_{i-1}^q}{\alpha(x_{i-1}^{q-1}+1)}.
\end{equation}
\end{definition}

Our goal is to show that the tower $\mathcal F^{(L)}$ is optimal for any prime $L$ different from $T$. Since one readily can show that the genera of the function fields occurring in $\mathcal F^{(L)}$ are the same as those of the corresponding function fields in $\mathcal F$, we can use Proposition \ref{prop:genusunreducedjnt} to estimate the genus. Before investigating the splitting locus of $\mathcal F^{(L)}$ we state a lemma.

\begin{lemma}\label{lem:galoisclosure}
The Galois closure of the extension $\FF(u_1)/\FF(j_1)$ is given by $\FF(x_0)/\FF(j_1)$.
Furthermore, let $L \in \F[T]$ be a prime different from $T$ of degree $d$, then the Galois closure of the extension $\FF_L^{(2)}(u_1)/\FF_L^{(2)}(j_1)$ is given by $\FF_L^{(2)}(x_0)/\FF_L^{(2)}(j_1)$.
\end{lemma}
\begin{proof}
In Section \ref{section:two} and the beginning of Section \ref{section:four}, the variables $j_1$ and $x_0$ were connected explicitly with each other in the following way:
\begin{equation}\label{eq:equations}
j_1=\dfrac{(u_1+T)^{q+1}}{u_1} \makebox{, } u_1=\frac{(v_0+T)^q}{v_0^{q-1}} \makebox{ and } v_0=-T(x_0^{q-1}+1).
\end{equation}
The first equation can be rewritten as $(u_1+T)^{q+1}-j_1(u_1+T)+j_1T$. Then it follows from \cite{abhyankhar,abhyankhar2} that the Galois closure of $\FF(u_1)/\FF(j_1)$ has Galois group $\PGL(2,\F)$ over $\FF(j_1)$. In order to show that the extension $\FF(u_1)/\FF(j_1)$ is Galois with Galois group $\PGL(2,\F)$, we write $j_1$ in terms of $x_0$ using Equation (\ref{eq:equations}). The result is $$j_1=-T^q\frac{(x_0^{q^2}-x_0)^{q+1}}{(x_0^q+x_0)^{q^2+1}}.$$
It now follows that the extension $\FF(x_0)/\FF(j_1)$ is a Galois extension with Galois group $\PGL(2,\F)$. The second part of the lemma can be shown in exactly the same way.
\end{proof}

Now we can show the main result of this section.

\begin{theorem}\label{thm:tower2}
The tower $\mathcal F^{(L)}$ is asymptotically optimal.
\end{theorem}
\begin{proof}
All that remains to be proved is that the polynomial $p_d^{(L)}(t^q+t)$ splits over $\mathbb{F}_{L}^{(2)}$. Let $\rho$ be a root of this polynomial in some extension field of $\F$ and write $\sigma=\rho^{q-1}$. We already know from Corollary \ref{cor:splitting2} that $Q_\sigma$, the zero of $u_1-\sigma$ in $\mathbb{F}_{L}^{(2)}(u_1)$, is rational and that it splits in the extension $\mathbb{F}_{L}^{(2)}(x_0^{q-1})/\mathbb{F}_{L}^{(2)}(u_1)$. However, from Lemma \ref{lem:galoisclosure} it follows that the extension $\mathbb{F}_{L}^{(2)}(x_0)/\mathbb{F}_{L}^{(2)}(u_1)$ is the Galois closure of the extension $\mathbb{F}_{L}^{(2)}(x_0^{q-1})/\mathbb{F}_{L}^{(2)}(u_1)$. Therefore, the place $Q_\sigma$ splits completely in $\mathbb{F}_{L}^{(2)}(x_0)/\mathbb{F}_{L}^{(2)}(u_1)$. This implies that $\rho \in \mathbb{F}_{L}^{(2)}$ as desired.
\end{proof}

Note that the proof of above theorem implies the roots of the polynomial $p_d^{(L)}(s)$ are $(q-1)$-st powers in $\FF_L^{(2)}$. The same is true for the roots of $p_d^{(L)}(s(s+1)^{q-1})$.

\bibliographystyle{99}

\end{document}